\documentclass[draft,reqno]{amsproc}
\usepackage{amsmath}
\usepackage{amsfonts}
\usepackage{mathrsfs}
\usepackage[dvips]{graphicx}
\usepackage{amssymb}
\usepackage{amsbsy}
\usepackage{amsthm}
\usepackage{graphics}
\usepackage{color}
\usepackage{textcomp}
\usepackage{verbatim}

\makeatletter
\@namedef{subjclassname@2010}{%
\textup{2010} Mathematics Subject
Classification} \makeatother

\numberwithin{equation}{section}

\newtheorem{theorem}{Theorem}[section]
\newtheorem*{theorem*}{Theorem}

\newtheorem*{prob*}{Problem}

\theoremstyle{remark}

\newtheorem{df}[theorem]{Definition}

\newcommand*{\ascr}{\mathscr{A}}

\newcommand*{\card}{\mathrm{card}}
\newcommand*{\hh}{\mathcal{H}}

\newcommand*{\kk}{\mathcal{K}}
\newcommand*{\lcal}{\mathcal{L}}
\newcommand*{\natu}{\mathbb{N}}
\newcommand*{\mcal}{\mathcal{M}}

\theoremstyle{definition}

\newcommand*{\Le}{\leqslant}

\newcommand*{\ran}{\mathscr{R}}

\newcommand*{\ogr}{\boldsymbol{B}}

\def\sslim{\qopname\relax m{{\sf{s}\textrm-}lim}}
\def\wwlim{\qopname\relax m{{\sf{w}\textrm-}lim}}

\begin{document}

   \title[Hyperrigidity III] {Hyperrigidity III}

   \author[P. Pietrzycki, M. Scherer and J. Stochel]{Pawe{\l} Pietrzycki, Marcel Scherer and  Jan Stochel}



   \address{Wydzia{\l} Matematyki i Informatyki, Uniwersytet
Jagiello\'{n}ski, ul. {\L}ojasiewicza 6,
PL-30348 Krak\'{o}w, Poland.}

   \email{pawel.pietrzycki@im.uj.edu.pl}

   \address{Technion Israel Institute of Technology, Technion City, Haifa, 3200003, Israel}
   \email{scherer@math.uni-sb.de, scherer@campus.technion.ac.il}

   \address{Wydzia{\l} Matematyki i Informatyki, Uniwersytet
Jagiello\'{n}ski, ul. {\L}ojasiewicza 6,
PL-30348 Krak\'{o}w, Poland.}

   \email{jan.stochel@im.uj.edu.pl}

%    \thanks{The research of both
% authors was supported by the National Science Center
% (NCN) Grant OPUS No.\ DEC-2021/43/B/ST1/01651.}

   \begin{abstract}In this paper, we study hyperrigidity for $C^*$-algebras. We will show that hyperrigidity can be expressed solely in terms of representations, without the need to involve general unital completely positive maps. 

   \end{abstract}

   \maketitle

   \section{Introduction}
Denote by $\natu$ the set of all positive integers.
Henceforth, $\ogr(\hh)$ stands for the $C^*$-algebra
of all bounded linear operators on a (complex) Hilbert
space $\hh$ and $I$ denotes the identity operator on
$\hh$. We always assume that completely positive maps
are linear and representations of unital
$C^*$-algebras preserve units and involutions.

% The classical approximation theorem due to P. P.
% Korovkin \cite{Kor53} states that for any sequence of
% positive linear maps $\varPhi_k\colon C[0,1]\to
% C[0,1]$ ($k\in \natu$),
%    \begin{align*}
% \lim_{k\to\infty}\|\varPhi_k(x^j)-x^j\|=0 \;\; \forall
% j\in \{0,1,2\} \implies
% \lim_{k\to\infty}\|\varPhi_k(f)-f\|=0 \;\; \forall
% f\in C[0,1].
%    \end{align*}
% In other words, the asymptotic behaviour of the
% sequence $\{\varPhi_k\}_{k=1}^\infty$ on the
% $C^*$-algebra $C[0, 1]$ is uniquely determined by the
% set $G=\{\boldsymbol{1}, x, x^2\}$. Sets with this
% property are called Korovkin sets. Korovkin's theorem
% unified many existing approximation processes such as
% the Bernstein polynomial approximation of continuous
% functions and the Fej\'{e}r trigonometric polynomial
% approximation of continuous functions on the unit
% circle. Another major achievement was the discovery of
% geometric theory of Korovkin sets by Y. A.
% \v{S}a\v{s}kin \cite{Sas67}. Namely, \v{S}a\v{s}kin
% observed that the key property of $G$ is that the
% Choquet boundary of the vector space spanned by $G$
% coincides with $[0, 1]$. Detailed surveys of most of these
% developments can be found in \cite{BL75,AC94,Alt10}.
% Recently, the \v{S}a\v{s}kin theorem has been extended
% by Davidson and Kennedy to the case of (not
% necessarily metrizable) compact Hausdorff spaces (see
% \cite[Theorem~5.3]{DK21}).

Motivated both by the
fundamental role of the Choquet boundary in classical
approximation theory, and by the importance of
approximation in the contemporary theory of operator
algebras, Arveson introduced hyperrigidity as a form
of 'noncommutative' approximation that captures many
important operator-algebraic phenomena.

We say that $G$ is a set of generators of a unital
$C^*$-algebra $\ascr$, or that $G$ generates $\ascr$,
if the smallest unital $C^*$-subalgebra of $\ascr$
containing $G$ is equal to $\ascr$.
   \begin{df}
A finite or countably infinite set $G$ of generators
of a unital $C^*$-algebra $\ascr$ is said to be
\textit{hyperrigid} if for any faithful representation
$\pi\colon \ascr\to \ogr(\hh)$ on a separable Hilbert
space $\hh$ and for any sequence $\varPhi_k\colon
\ogr(\hh)\to \ogr(\hh)$ ($k\in \natu$) of unital
completely positive (UCP) maps:
   \allowdisplaybreaks
   \begin{align*}
\lim_{k\to\infty}\|\varPhi_k(\pi(g))-\pi(g)\|=0
\; \forall g\in G \implies
\lim_{k\to\infty}\|\varPhi_k(\pi(a))-\pi(a)\|=0
\; \forall a\in \ascr.
   \end{align*}
   \end{df}
Arveson gave a
characterization of hyperrigidity that replaces the
limit process by the so-called {\em unique extension
property} (see Theorem~\ref{arahc}(iii)). On the other hand, Kleski provided a characterization of hyperrigidity in which norm topology is replaced by weak operator topology (equivalently  strong operator topology) (see Theorem~\ref{arahc}(iv)).  
  \begin{theorem}[\mbox{\cite[Theorem~2.1]{Arv11},
\cite[Proposition~2.2]{Kle14} \& \cite[Theorem~B.2]{P-S-24}}] \label{murzeqiv}\label{arahc}
Suppose that $ G$ is a nonempty subset of a unital
$C^*$-algebra $\ascr$. Consider the following
conditions{\em :}
   \begin{enumerate}
   \item[(i)] $G$ is hyperrigid,
   \item[(ii)] for every Hilbert space $\hh$, every
representation $\pi\colon \ascr\to \ogr(\hh)$ and
every sequence $\varPhi_n\colon \ascr\to \ogr(\hh)$
$($$n\in \natu$$)$ of UCP maps,
   \begin{align*}
\lim_{n\to\infty}\|\varPhi_n(g)-\pi(g)\|=0 \;\;
\forall g\in G \implies
\lim_{n\to\infty}\|\varPhi_n(a)-\pi(a)\|=0 \;\;
\forall a\in \ascr,
   \end{align*}
   \item[(iii)] for every Hilbert space $\hh$ and every
representation $\pi\colon \ascr \to \ogr(\hh)$,
$\pi|_{ G}$ has the unique extension property,
%    \item[(iv)] for every unital  $C^*$-algebra
% $\bscr$, every unital $*$-homomorphism of
% $C^*$-alge\-bras $\theta\colon \ascr \to \bscr$ and
% every UCP map $\varPhi\colon \bscr \to \bscr$,
%    \begin{align*}
% \varPhi(x)=x \; \forall x \in \theta( G) \implies
% \varPhi(x)=x \; \forall x \in \theta(\ascr),
%    \end{align*}
   \item[(iv)] for every Hilbert space $\hh$,
every representation $\pi\colon \ascr\to \ogr(\hh)$
and every sequence $\varPhi_n\colon \ascr\to
\ogr(\hh)$ $($$n\in \natu$$)$ of UCP maps,
   \begin{align*}
\wwlim_{n\to\infty}\varPhi_n(g)= \pi(g) \; \forall
g\in G \implies \sslim_{n\to\infty}\varPhi_n(a)=\pi(a)
\; \forall a\in \ascr.
   \end{align*}
   \end{enumerate}
Then conditions {\em (i)}-{\em (iii)} are equivalent.
Moreover, if $\ascr$ is separable, then conditions
{\em (i)}-{\em (iv)} are equivalent, and what is more
they are still equivalent regardless of whether the
Hilbert spaces considered in either of them are
separable or not.
   \end{theorem}

In accordance with \v{S}a\v{s}kin's insightful
observation, Arveson \cite{Arv11} conjectured that
hyperrigidity is equivalent to the noncommutative
Choquet boundary of $ G$ being as large as possible,
in the sense that every irreducible representation of
$C^*(G)$, the unital $C^*$-algebra generated by $G$,
should be a boundary representation for $ G$. This is
now known as \textit{Arveson's hyperrigidity
conjecture} (see \cite[Conjecture~4.3]{Arv11}). Some
positive solutions of Arveson's hyperrigidity
conjecture were found for certain classes of
$C^*$-algebras (see
\cite{Arv11,Kle14,CH18,Sal19,Ka-Sh19,Sh20,P-S-24,Sch24}). In full
generality, Arveson's hyperrigidity conjecture turns
out to have a negative solution. Recently, a
counterexample has been found by B. Bilich and A.
Dor-On in \cite{Bi-Dor24} (see also \cite{Bi24}).
However, Arveson's hyperrigidity conjecture is still
open for commutative C*-algebras (even the singly
generated case is not resolved).

In recent years, this issue has attracted considerable
interest in various parts of operator algebras and
operator theory
\cite{KS15,Clo18,Clo18b,CH18,Har19,DK19,Ka-Ra20,DK21,Cl-Ti21,Ki21,Thom24,CH-Th24,P-S-24II}.
In particular, it is related to the Arveson-Douglas
essential normality conjecture involving quotient
modules of the Drury-Arveson space \cite[Theorem
4.12]{KS15} as well as is closely related to
the problem of characterizing the spectrality of semispectral measures by their
operator-moments (see \cite{P-S-24}).

\section{Hyperrigidity via representations}
  In this section, we  show that the property of hyperrigidity can be expressed solely in terms of representations, without the need to involve general UCP maps. 
  % The proof is inspired by \cite{Brown16}
  
  \begin{theorem}\label{mainth}
Suppose that $G$ is a nonempty subset of a unital separable
$C^*$-algebra $\ascr$. Then the following conditions are equivalent{\em :}
   \begin{enumerate}
   \item[(i)] $G$ is hyperrigid,
   \item[(ii)] for every Hilbert space $\hh$,
every representations $\pi\colon \ascr\to \ogr(\hh)$
and  $\pi_n\colon \ascr\to
\ogr(\hh)$ $($$n\in \natu$$)$,
   \begin{align*}
\wwlim_{n\to\infty}\pi_n(g)= \pi(g) \; \forall
g\in G \implies \sslim_{n\to\infty}\pi_n(a)=\pi(a)
\; \forall a\in \ascr.
   \end{align*}
   \end{enumerate}
   \end{theorem}
   \begin{proof}[First proof of Theorem~\ref{mainth}]
   (i)$\Rightarrow$(ii) It follows from Theorem~\ref{murzeqiv}.
   
   (ii)$\Rightarrow$(i)
       First note that there exists a representation $\sigma\colon \ascr\to \ogr(\lcal)$, where $\lcal$ is an infinite-dimensional separable Hilbert space, i.e. $\dim \lcal=\aleph_0$. Indeed, let $\tilde{\sigma}\colon \ascr\to \ogr(\mcal)$ be a faithful representation of $\ascr$. Since, $\ascr$ is separable, so is $\mcal$. If $\dim \mcal=\aleph_0$ we are done. If not, then $\sigma:=\tilde{\sigma}\oplus\tilde{\sigma}\oplus\tilde{\sigma}\oplus...\colon \ascr\to \ogr(\lcal)$ is the required (faithfull) representation, where $\lcal=\mcal\oplus\mcal\oplus\dots$.
       
        Let $\pi\colon \ascr\to
\ogr(\hh)$ be a representation and
$\varPhi\colon \ascr \to \ogr(\hh)$ be a UCP
map such that
   \begin{align} \label{wuir}
\pi(g) = \varPhi(g), \quad g\in G.
   \end{align}
It follows from the Stinespring dilation theorem (see
\cite[Theorem~1]{Sti55}) that there exist a Hilbert
space $\kk\supseteq \hh$ and a representation
$\rho\colon \ascr \to \ogr(\kk)$ such that
   \begin{align} \label{viro}
\varPhi(a)= P\rho(a)|_{\hh}, \quad a \in
\ascr,
   \end{align}
where $P\in \ogr(\kk)$ is the orthogonal
projection of $\kk$ onto $\hh$. 

% Without loss of generality, we can assume that $\dim \hh =\dim \kk=\aleph_0$ and $\dim \kk\ominus \hh=\aleph_0$.
Now we will show that without loss of generality we can assume that $\dim \hh = \aleph_0$.
% Without loss of generality supousse that $\dim \hh = \aleph_0$.
We always have that $\dim \hh \Le \aleph_0$ as a consequence of the separability of $\ascr$.
% We can always assume that $\dim \hh \Le\aleph_0$ (because of seperability of $\ascr$).
Suppose $\dim \hh <\aleph_0$. Consider $\tilde \hh =\hh\oplus \lcal$ (recall $\dim \lcal=\aleph_0$), $\tilde{\pi}:=\pi\oplus\sigma\colon \ascr \to \ogr(\tilde{\hh})$ is a faithful representation and $\tilde{\varPsi}:=\varPsi\oplus\sigma$ is a UCP map such that $\tilde{\pi}|_G=\tilde{\varPsi}|_G$. Since $\dim \tilde{H}=\aleph_0$, we infer from the $\aleph_0$-version of the proof that $\tilde{\pi}=\tilde{\varPsi}$ on $\ascr$, i.e., $\pi\oplus\sigma=\varPhi\oplus\sigma$, which implies that $\pi=\varPhi$, and we are done.

Our next goal is to show that, without loss of generality $\dim \kk=\aleph_0$.
For this set $\kk_{\text{min}}=\bigvee_{a\in \ascr}\rho(a)\hh$. Then $\kk_{\text{min}}$ reduces $\rho$ and $\hh\subseteq\kk_{\text{min}}$, so $\varPhi=P_m\rho_m|_\hh$, where $P_m\in \ogr(\kk_{\text{min}})$ is the orthogonal projection of $\kk_{\text{min}}$ onto $\hh$, and $\rho_m=\rho|_{\kk_{\text{min}}}$. However, if $\bar{\ascr_0}=\ascr$ with $\card \ascr_0\Le\aleph_0$, and $\bar{\hh_0}=\hh$ with $\card \hh_0\Le\aleph_0$ then
\begin{align*}
   \kk_{\text{min}}=\bigvee_{a\in \ascr_0}\rho(a)\hh =\bigvee_{a\in \ascr_0}\rho(a)\hh_0,
\end{align*}
so $\dim \kk_{\text{min}}\Le\aleph_0$. However, $\aleph_0=\dim \hh\Le\dim \kk_{\text{min}}\Le\aleph_0$, so $\dim \kk_{\text{min}}=\aleph_0$
Our next step is to show that, without loss of generality, we can assume that $\dim \kk\ominus\hh=\aleph_0$. Clearly $\dim \kk\ominus\hh\Le \dim \kk =\aleph_0$. Suppose $\dim \kk\ominus\hh<\aleph_0$. Consider $\tilde{\kk}=\kk\oplus\lcal$, and $\tilde{\rho}=\rho\oplus\sigma\colon \ascr \to \ogr(\kk\oplus\lcal)$. Then $\tilde{\rho}$ is a faithfull representation and $\dim \kk\oplus\lcal=\aleph_0$ and $\varPsi=P^{\kk\oplus\lcal}_\hh\tilde{\rho}|_\hh$ on $\ascr$.

($P^{\kk\oplus\lcal}_\hh$ is the orthogonal projection of $\kk\oplus\lcal$ onto $\hh$)

Note that, without loos of generality we can assume that
$\hh=\lcal$ and $\kk\oplus\hh=\lcal\oplus\lcal\oplus\lcal\oplus\dots$, where as before $\dim \lcal =\aleph_0$.

Let $S\in\ogr(\kk\ominus\hh)$ be the isometric shift of multiplicity $\aleph_0$, i.e.,
\begin{align*}
    S(h_0\oplus h_1\oplus\dots)=0\oplus h_0\oplus h_1\oplus\dots,\quad h_j\in \lcal.
\end{align*}
Set $V=I_\hh\oplus S\in \ogr(\kk)$. Then $V$ is an isometry such that $\sslim_{n\to \infty} V^{*n}=P$, where $P\in\ogr(\kk)$ is the orthogonal projection onto $\hh$.

Define the representation $\sigma_\infty\colon\ascr\to \ogr(\kk)$ by
$\sigma_\infty=\sigma\oplus\sigma\oplus\dots$, where
$\kk=\hh\oplus(\kk\ominus\hh)=\lcal\oplus(\lcal\oplus\lcal\oplus\dots)$

Note that 
\begin{align*}
    (I_\kk-V^mV^{*m})(f\oplus(h_0\oplus h_1\oplus\dots))=0\oplus(h_0\oplus h_1\oplus\dots\oplus h_{m-1}\oplus0\oplus\dots),
\end{align*}
so $I_\kk-V^mV^{*m}$ is diagonal with respect to the same orthogonal decomposition $\kk=\lcal\oplus(\lcal\oplus\lcal\oplus\dots)$.
Also $\sigma_\infty(a)$ is diagonal with respect to the same orthogonal decomposition. Thus
\begin{align*}
    (I_\kk-V^mV^{*m})\sigma_\infty(a)=\sigma_\infty(a)(I_\kk-V^mV^{*m})
\end{align*}
for all $m\in\mathbb{Z}_+$ and $a\in \ascr$.
 
   This implies that $\ascr\ni a\to(I_\kk-V^mV^{*m})\sigma_\infty(a)\in \ogr(\kk)$ is $*-$multiplicative and linear. In fact, we have
   \begin{align*}
       (I_\kk-V^mV^{*m})\sigma_\infty(a)=0\oplus\underbrace{\sigma(a)\oplus\dots\sigma(a)}_{m}\oplus0\oplus\dots
   \end{align*}
Define  $\pi_m\colon\ascr\to\ogr(\kk)$
\begin{align*}
\pi_m(a)=V^m\rho(a)V^{*m}+(I_\kk-V^mV^{*m})\sigma_\infty(a),\quad a \in \ascr.
\end{align*}
Note that
\begin{align*}
\pi_m(a)=\psi_m(a)|_{\ran(V^m)}\oplus\underbrace{\sigma(a)\oplus\sigma(a)\oplus\dots\sigma(a)}_{m},
\end{align*}
where
\begin{align*}
\psi_m(a)=V^m\rho(a)V^{*m},\quad a \in \ascr.
\end{align*}
As we know, $\pi_m\colon \ascr\to\ogr(\kk)$ is linear and $*-$multiplicative, so 
$\psi_m(a)|_{\ran(V^m)}$ is a representation. This implies that each $\pi_m$ is a representation. Since, we have $\sslim_{n\to \infty} V^{*n}=P$, we get 
\begin{align}\label{wigi}
\wwlim_{n\to\infty}\varPhi_n(a)&=\wwlim_{n\to\infty}V^m\rho(a)V^{*m}+\wwlim_{n\to\infty}(I_\kk-V^mV^{*m})\sigma_\infty(a)\\&=P\rho(a)P+(I-P)\sigma_\infty(a)=P\rho(a)|_\hh\oplus\sigma(a)\oplus\sigma(a)\oplus\dots,\quad a\in \ascr.\notag
\end{align}
By \eqref{wigi}, we get
\begin{align*}
     \wwlim_{n\to\infty} \pi_n(g)=\pi(g)\oplus\sigma(g)\oplus\sigma(g)\oplus\dots,\quad g\in G.
\end{align*}
It follows from Theorem~\ref{murzeqiv}(v) that
\begin{align}\label{golret}
\wwlim_{n\to\infty} \varPhi_n(a)=\pi(a)\oplus\sigma(a)\oplus\sigma(a)\oplus\dots,\quad a\in \ascr.
\end{align}
Comparing \eqref{wigi} with \eqref{golret} we obtain
\begin{align*}
    \pi(a)=P\rho(a)|_\hh,\quad a\in \ascr.
\end{align*}
In view of beging of the proof, we conclude that $\pi(a)=\varPhi(a)$ for all $a\in \ascr$. Applying Theorem~\ref{murzeqiv}, we obtain that $G$ is hyperrigid set.
  \end{proof}

\begin{proof}[Second proof of Theorem~\ref{mainth}]
(ii)$\Rightarrow$(i)
        Let $\pi\colon \ascr\to
\ogr(\hh)$ be a representation and
$\varPhi\colon \ascr \to \ogr(\hh)$ be a UCP
map such that
 \begin{align}\label{eqong}
\pi(g) = \varPhi(g), \quad g\in G.
   \end{align}
It follows from the Stinespring dilation theorem (see
\cite[Theorem~1]{Sti55}) that there exist a Hilbert
space $\kk\supseteq \hh$ and a representation
$\rho\colon \ascr \to \ogr(\kk)$ such that
   \begin{align} \label{viro}
\varPhi(a)= P\rho(a)|_{\hh}, \quad a \in
\ascr,
   \end{align}
where $P\in \ogr(\kk)$ is the orthogonal
projection of $\kk$ onto $\hh$. 
In the same way as in the first proof Theorem~\ref{mainth} we can assume without loss off generality that
\begin{align}\label{eqdim}
    \dim \hh = \dim \kk = \aleph_0.
\end{align}
Let $(e_n)_{n\in\natu}$ be an orthonormal basis of $\hh$ and set
\[
F_n=\operatorname{span}\{e_1,\ldots,e_n\}.
\]
Then
\[
F_1\subseteq F_2\subseteq\cdots
\qquad\text{and}\qquad
\overline{\bigcup_{n=1}^\infty F_n}=\hh.
\]
By \eqref{eqdim}, $\dim(\hh\ominus F_n)=\dim(\kk\ominus F_n)$
for every $n\in\natu$. Hence there exists a unitary operator $
U_n\colon\hh\to\kk$
such that
\[
U_n|_{F_n}=I_{F_n}.
\]
We claim that $\lim_{n\to\infty}U_n f=f$, for $f\in\hh$.
Indeed, if $P_n$ denotes the orthogonal projection of $\hh$
onto $F_n$, then
$\lim_{n\to\infty}P_n f= f$. Since $U_nP_nf=P_nf$, we have
\begin{align*}
\|U_nf-f\|
&\leq
\|U_n(f-P_nf)\|+\|P_nf-f\|\\
&=
2\|f-P_nf\|.
\end{align*}
Therefore passing with $n$ to infinity, we get $\lim_{n\to\infty}U_n f=f$, for $f\in\hh$.

Define
\[
\pi_n(a)=U_n^*\rho(a)U_n,
\qquad a\in\ascr.
\]
Then each $\pi_n\colon\ascr\to\ogr(\hh)$ is a representation.

For $a\in\ascr$ and $f,f'\in\hh$, we have
\begin{align*}
\lim_{n\to\infty}\langle\pi_n(a)f,f'\rangle
&=
\lim_{n\to\infty}\langle\rho(a)U_nf,U_nf'\rangle\\
&=
\langle\rho(a)f,f'\rangle\\
&=
\langle P\rho(a)|_{\hh}f,f'\rangle\\
&=
\langle\varPhi(a)f,f'\rangle.
\end{align*}
Therefore
\begin{align}\label{qoop}
    \wwlim_{n\to\infty}\pi_n(a)=\varPhi(a),
\qquad a\in\ascr.
\end{align}
In particular, by \eqref{eqong}
\[
\wwlim_{n\to\infty}\pi_n(g)
=
\varPhi(g)
=
\pi(g),
\qquad g\in G.
\]
we obtain
\[
\wwlim_{n\to\infty}\pi_n(g)=\pi(g),
\qquad g\in G.
\]
Hence, by assumption \emph{(ii)},
\[
\sslim_{n\to\infty}\pi_n(a)=\pi(a),
\qquad a\in\ascr.
\]
Consequently,
\[
\wwlim_{n\to\infty}\pi_n(a)=\pi(a),
\qquad a\in\ascr.
\]
This combined with \eqref{qoop} gives $\varPhi=\pi$.

Therefore \(\pi\) has the unique extension property relative to \(G\).
Since \(\pi\) was arbitrary, \(G\) is hyperrigid.

\end{proof}
 
%   \begin{theorem} \label{idyln}
%Let $G$ be a hyperrigid subset of a unital
%$C^*$-algebra $\ascr$ and $\varPhi\colon \ascr \to
%\ogr(\hh)$ be a UCP map such that $G\subseteq \ker
%\varPhi$. Let $\hh$ and $\kk$ be Hilbert spaces,
%$\pi\colon \ascr \to \ogr(\hh)$ and $\rho\colon \ascr
%\to \ogr(\kk)$ be representations and $R\in
%\ogr(\hh,\kk)$ be a contraction such that $\pi(g) =
%R^*\rho(g)R$ for every $g\in G$. Then
%   \begin{align} \label{fewas}
%R \pi(a) = \rho(a) R \;\; \forall a\in \ascr \;\; \&
%\;\; \ker \varPhi \subseteq J \;\; \& \;\; J \text{ is
%a $*$-ideal in } \ascr,
%   \end{align}
%where $J:=\{a\in \ascr \colon \pi(a) = R^*\rho(a)R\}$.
%Moreover, if $R$ is non-isometric, then there exists a
%character $\chi$ of $\ascr$ such that
%   \begin{enumerate}
%   \item[(i)] $\ker \chi=J$ and $\varPhi = \chi I$,
%   \item[(ii)] $\triangle$ commutes with $\pi$
%and $\triangle_*$ commutes with $\rho$,
%   \item[(iii)] $\hh_1:=\overline{\ran(\triangle)}\neq \{0\}$
%reduces $\pi$ and $\kk_1:
%=\overline{\ran(\triangle_*R)}$ reduces $\rho$,
%   \item[(iv)] $R=R_0\oplus R_1$, where $R_0\in
%\ogr(\hh_0, \kk_0)$ is an isometry with $\hh_0:=
%\hh_1^{\perp}$ and $\kk_0:= \kk_1^{\perp}$, and
%$R_1\in \ogr(\hh_1, \kk_1)$ is a pure contraction with
%dense range,
%   \item[(v)] $\pi=\pi_0 \oplus \chi
%I_{\hh_1}$ and $\pi_0=R_0^*\rho_0R_0$, where
%$\pi_0:=\pi|_{\hh_0}$ and $\rho_0:=\rho|_{\kk_0}$,
%   \item[(vi)] $\rho=\rho_0 \oplus \chi
%I_{\kk_1}$.
%   \end{enumerate}
%   \end{theorem}

   \bibliographystyle{amsalpha}
   
   \end{document}